\documentclass[paper=a4]{scrartcl}
\usepackage{lmodern}
\usepackage{amsmath,amssymb,amsfonts,amsthm}
\usepackage{textcomp}
\usepackage{graphicx}
\usepackage{enumitem}
\setlist{noitemsep,leftmargin=*}
\usepackage[T1]{fontenc}
\usepackage[utf8]{inputenc}
\usepackage{lmodern}
\usepackage{microtype}
\usepackage[sort,compress,numbers]{natbib}
\usepackage{hyperref}
\hypersetup{%
  pdfauthor={Gábor Hegedüs, Josef Schicho, and Hans-Peter Schröcker},%
  pdftitle={The Theory of Bonds: A New Method for the Analysis of Linkages},%
  pdfborder={0 0 0},%
  colorlinks=false,%
}
\theoremstyle{plain}
\newtheorem{lem}{Lemma}
\newtheorem{prop}[lem]{Proposition}
\newtheorem{thm}[lem]{Theorem}
\newtheorem{cor}[lem]{Corollary}
\theoremstyle{definition}
\newtheorem{defn}[lem]{Definition}
\newtheorem{example}{Example}
\theoremstyle{remark}
\newtheorem{rem}[lem]{Remark}

\def\P{\mathbb P}
\def\N{\mathbb N}
\def\Z{\mathbb Z}
\def\C{\mathbb C}
\def\D{\mathbb D}
\def\R{\mathbb R}
\def\H{\mathbb H}
\def\eps{\epsilon}

\DeclareMathOperator{\ord}{ord}
\DeclareMathOperator{\norm}{N}
\DeclareMathOperator{\normalization}{NC}

\newcommand{\SO}[1][3]{\mathrm{SO}_{#1}}
\newcommand{\SE}[1][3]{\mathrm{SE}_{#1}}
\newcommand{\ci}{\mathrm{i}}
\newcommand{\qi}{\mathbf{i}}
\newcommand{\qj}{\mathbf{j}}
\newcommand{\qk}{\mathbf{k}}

\renewcommand{\bar}[1]{\overline{#1}}
\newcommand{\eoex}{$\diamond$}

\addtokomafont{sectioning}{\rmfamily}
\setkomafont{descriptionlabel}{\itshape}

\title{The Theory of Bonds: A New Method for the Analysis of Linkages}
\author{Gábor Hegedüs\thanks{Obuda University, John von Neumann Faculty on Informatic,
Applied Mathematical Institute, Antal Bejczy Center for Intelligent Robotics,
    \href{mailto:greece@math.bme.hu}{greece@math.bme.hu}}, %
  Josef Schicho\thanks{Johann Radon Institute for Computational and
    Applied Mathematics, Austrian Academy of Sciences (RICAM),
    \url{http://www.risc.jku.at/people/jschicho/},
    \href{mailto:josef.schicho@oeaw.ac.at}{josef.schicho@oeaw.ac.at}}, %
  and Hans-Peter Schröcker\thanks{University Innsbruck, Unit
    Geometry and CAD,
    \url{http://geometrie.uibk.ac.at/schroecker/},
    \href{mailto:hans-peter.schroecker@uibk.ac.at}{hans-peter.schroecker@uibk.ac.at}}}
\begin{document}

\maketitle
\begin{abstract}
  In this paper we introduce a new technique, based on dual
  quaternions, for the analysis of closed linkages with revolute
  joints: the theory of bonds. The bond structure comprises a lot of
  information on closed revolute chains with a one-parametric
  mobility. We demonstrate the usefulness of bond theory by giving a
  new and transparent proof for the well-known classification of
  overconstrained 5R linkages.
\end{abstract}



\section*{Introduction}

In this paper we rigorously develop the theory of bonds \cite{HSS3}, a
tool for the analysis of closed linkages with revolute joints and one
degree of freedom. The configuration curve of such a linkage can be
described by algebraic equations. Intuitively, bonds are points in the
configuration curve with complex coefficients where something
degenerate happens. For a typical bond of a closed $n$R chain, there
are exactly two joints with degenerate rotation angles (see
Theorem~\ref{thm:i} and the subsequent remark). In this way, the bond
``connects'' the two links. It is remarkable that a lot of information
on the linkage can be extracted from this combinatorial behavior of
the bonds.

In order to describe the forward kinematic map from the configuration
curve into the group of Euclidean displacements, we use the language
of dual quaternions (see also \cite{Blaschke:60,Hao,husty10,selig05}).
For any pair of links, the set of possible relative poses is a curve
on the Study quadric in the projective space $\P^7$. In
Theorem~\ref{thm:bond} we compute the degrees of these curves by the
combinatorial behavior of the bonds.

Theorem~\ref{thm:coup} is interesting in its own right. It relates the
geometry of three consecutive revolute axes with the dimension of a
certain linear subspace of the 8-dimensional vector space $\D\H$ of
dual quaternions. In general, this subspace is equal to $\D\H$, but,
it may also be of dimension 4 or 6 for particular positions of the
three lines. More precisely, the dimension is 4 if and only if the
three lines are parallel or meet in a common point, and it is 6 if and
only if the lines appear as revolute axes in a 4R Bennett linkage. We
show that a bond of certain type appears if and only if the dimension
of the corresponding linear subspace is less than~8.

Section~\ref{sec:bonds} features a rigorous definition of bonds and
connection numbers. We visualize the latter in bond diagrams and show
how to read off linkage properties. As an application of bond theory,
we give a new proof of Karger's classification theorem for
overconstrained closed 5R-linkages \cite{Karger} in
Section~\ref{sec:classification}. In contrast to Karger's original
proof, it does not require the aid of a computer algebra system.

We announced simplified versions of the results of this paper in
\cite{HSS3} without proofs. %
Supplementary material to this article can be found on the accompanying web-site %
\url{http://geometrie.uibk.ac.at/schroecker/bonds/}.

\section{Dual quaternions}
\label{sec:dual-quaternions}

In this section, we recall the well-known and classical description of
the group of Euclidean displacements by dual quaternions; it is almost
identical to \cite[Section~2]{HSS2}. We just include it here to make
this paper more self-contained. More complete references are
\cite{Blaschke:60,Hao,husty10,selig05}.

We denote by $\SE$ the group of direct Euclidean displacements, i.e.,
the group of maps from $\R^3$ to itself that preserve distances and
orientation. It is well-known that $\SE$ is a semidirect product of
the translation subgroup and the orthogonal group $\SO$, which may be
identified with the stabilizer of a single point.

We denote by $\D:=\R+\eps\R$ the ring of dual numbers, with
multiplication defined by $\eps^2=0$. The algebra $\H$ is the
non-commutative algebra of quaternions, and $\D\H$ is the algebra of
quaternions with coefficients in $\D$. Every dual quaternion has a
primal and a dual part (both quaternions in $\H$), a scalar part in
$\D$ and a vectorial part in $\D^3$. The conjugate dual quaternion
$\bar{h}$ of $h$ is obtained by multiplying the vectorial part of $h$
by $-1$. The dual numbers $\norm(h) = h\bar{h}$ and $h+\bar{h}$ are
called the \emph{norm} and \emph{trace} of $h$, respectively.

By projectivizing $\D\H$ as a real 8-dimensional vectorspace, we
obtain $\P^7$. The condition that $\norm(h)$ is strictly real, i.e.\
its dual part is zero, is a homogeneous quadratic equation. Its zero
set, denoted by $S$, is called the Study quadric. The linear 3-space
represented by all dual quaternions with zero primal part is denoted
by $E$. It is contained in the Study quadric. The complement $S-E$ can
be identified with $\SE$. The primal part describes $\SO$.
Translations correspond to dual quaternions with primal part $\pm 1$
and strictly vectorial dual part. More precisely, the group
isomorphism is given by sending $h=p+\eps q$ to the map
\begin{equation*}
  \R^3\to\R^3,\quad
  v\mapsto\frac{pv\bar{p}+p\bar{q}-q\bar{p}}{p\bar{p}}.
\end{equation*}
(see \cite[p.~48]{Blaschke:60} or \cite[Section~2.4]{husty10}).

A nonzero dual quaternion represents a rotation if and only if its
norm and trace are strictly real and its primal vectorial part is
nonzero. It represents a translation if and only if its norm and trace
are strictly real and its primal vectorial part is zero. The
1-parameter rotation subgroups with fixed axis and the 1-parameter
translation subgroups with fixed direction can be geometrically
characterized as the lines on $S$ through the identity element $1$.
Among them, translations are those lines that meet the exceptional
3-plane~$E$.

\section{Linkages}
\label{sec:linkages}

In this section, we introduce some terminology on linkages, like
coupling curves and coupling spaces (relative motions between links,
described in terms of dual quaternions and linear spans of these
curves), and prove a useful theorem about the dimension of coupling
spaces.

We describe an open chain of $n > 0$ revolute joints by a sequence $L
= (h_1,\ldots,h_n)$ of unit dual quaternions $h_1,\ldots,h_n$ of zero
scalar part. Algebraically, this means that $h_i\bar{h_i} = -h_i^2 =
1$. Geometrically, we represent a revolute joint by a half-turn (a
rotation by the angle $\pi$). The group parametrized by
$(t-h_i)_{t\in\P^1}$ -- the parameter $t$ determines the rotation
angle -- is the group of the $(i+1)$-th link relative to the $i$-th
link. The position of the last link with respect to the first link is
then given by a product $(t_1-h_1)(t_2-h_2)\cdots(t_n-h_n)$, with
$t_1,\dots,t_n\in\P^1$. For a closed chain, we have the closure
condition
\begin{equation}
  \label{eq:1}
  (t_1-h_1)(t_2-h_2)\cdots(t_n-h_n) \in \R \setminus \{0\}.
\end{equation}
We view closed chains as cyclic sequences $L = (h_1,\ldots,h_n)$ and
we reflect this in the notational convention $h_{kn+i} := h_i$ for $k
\in \Z$.

\begin{defn}
  For a closed chain of revolute joints as described above, the set
  $K$ of all $n$-tuples $(t_1,\dots,t_n)\in(\P^1)^n$ fulfilling
  \eqref{eq:1} is called the chain's \emph{configuration set.}
\end{defn}

The dimension of the configuration set is called the \emph{degree of
  freedom} or the \emph{mobility} of the linkage. In this paper we
consider linkages of mobility one. This already implies $4 \le n \le
7$. For $n = 4$, we obtain planar, spherical or spatial four bar
linkages. The latter are usually referred to as Bennett linkages
\cite[Chapter~10, Section~5]{hunt78}. In general, closed chains of $n
< 7$ revolute joints are rigid. Thus, our results in this paper refer
to planar and spherical four bar linkages, to linkages of paradoxical
mobility with less than seven joints, and to linkages with seven
joints and one degree of freedom.

A linkage is a set of links, a set of joints, and a relation between
them, which we call ``attachment''. Any link has at least one attached
joint, and any joint has at least two attached links. If two joints
are attached to two links, then either the two joints or the two links
are equal. The link diagram is a linear hypergraph \cite[Ch.~1,
\S~2]{berge89} whose vertices are the links and whose hyperedges are
the joints; dually, the joint diagram is a linear hypergraph whose
vertices are the joints and whose hyperedges are the links. In both
cases, hyperedges are needed because a link can have more than two
attached joints and a joint can be attached to more than two links. In
this paper we will mostly be concerned with open and closed chains
with revolute joints. Here the two hypergraphs are just simple graphs,
consisting of a path or cycle. Nonetheless, it should be kept in mind
that the theory we develop can also be applied to cycles in general
linkages.

To each revolute joint we attach its axis of rotation (a line in
$\R^3$). It can be represented by the same dual quaternion $h_i$ as
the joint. This is almost the same as the representation of lines by
normalized Plücker coordinates which are composed of primal part and
negative dual part. The line determines $h_i$ up to multiplication
with $-1$. A configuration of a linkage consists of the specification
of suitable revolute angles for each pair of links joined by a joint.
This angle corresponds to a rotation of the form $t_i-h_i$,
$t_i\in\R$, or to the identity $1$ for $t_i=\infty$.

Let $L=(h_1,\dots,h_n)$ be a closed $n$R chain with mobility one. We
denote the links by $o_1,\dots,o_n$, and use the convention that $o_i$
is the link with joint axes $h_i,h_{i+1}$ for $i=1,\dots,n$. We use
$[n]$ as shorthand notation for the set $\{1,\ldots,n\}$. For $i<j \in
[n]$, we define the polynomial
\begin{equation}
  \label{eq:2}
  \begin{gathered}
    F_{i,j} = (t_{i+1}-h_{i+1})(t_{i+2}-h_{i+2})\cdots(t_j-h_j)
    \in \D\H[t_i,\ldots,t_j] \subset \D\H[t_1,\ldots,t_n],
  \end{gathered}
\end{equation}
and the map
\begin{equation}
  \label{eq:3}
  \begin{aligned}
  f_{i,j} \colon K & \to \SE, \\
  (t_1,\ldots,t_n) = \tau & \mapsto
  \begin{cases}
    F_{i,j}(\tau) & \text{if $F_{i,j}(\tau) \neq 0$,} \\
    \displaystyle
    \lim_{\tau'\to \tau} F_{i,j}(\tau') & \text{else.}
  \end{cases}
  \end{aligned}
\end{equation}
The distinction between the polynomial $F_{i,j}$ and the map $f_{i,j}$
is necessary because $F_{i,j}(\tau)$ may vanish at isolated points
$\tau \in K$ (see Corollary~\ref{cor:prod} and Example~\ref{ex:5}),
that is, the evaluation of $F_{i,j}$ at points $\tau \in K$ does not
give a well-defined map into $\P^7$. On the other hand, the map
$f_{i,j}$ is well-defined for all regular points $\tau\in K$. (Thus,
it should actually be defined on the normalization $\normalization(K)$
of $K$, compare Section~3.1.)

Because of the closure condition \eqref{eq:1}, we also have
\begin{equation}
  \label{eq:4}
  f_{i,j}(\tau) =
  \begin{cases}
    G_{i,j}(\tau) & \text{if $G_{i,j} \neq 0$,} \\
    \displaystyle \lim_{\tau'\to\tau} G_{i,j}(\tau') & \text{else,}
  \end{cases}
\end{equation}
where
\begin{gather*}
  G_{i,j} = \prod_{k=1}^{n+i-j} ( \bar{t_{i-k+1}-h_{i-k+1}})
  = (t_i+h_i)(t_{i-1}+h_{i-1})\cdots(t_{j+1}+h_{j+1}).
\end{gather*}
Note that the overline denotes dual quaternion conjugation and
$\overline{h} = -h$ whenever $h$ is of zero scalar part. Note further
that we define the product symbol $\prod$ by the recursion
$\prod_{k=1}^n x_k = (\prod_{k=1}^{n-1} x_k)x_n$ where $x_1,\ldots,x_n$
are elements of a non-commutative ring.

\begin{defn}
  \label{def:coupler}
  The map $f_{i,j}$ defined in \eqref{eq:3} is called the
  \emph{coupling map.} Its image is the \emph{coupling curve $C_{i,j}$.}
\end{defn}

The coupling curve $C_{i,j}$ describes the motion of link $o_j$
relative to link $o_i$. This is the reason for the seemingly strange
index convention in the definition of the coupling map~$f_{i,j}$.

\begin{defn}
  \label{def:3}
  For a sequence $h_i,h_{i+1},\dots,h_j$ of consecutive joints, we
  define the \emph{coupling space} $L_{i,i+1,\dots,j}$ as the linear
  subspace of $\R^8$ generated by all products $h_{k_1}\cdots
  h_{k_s}$, $i\le k_1<\cdots <k_s\le j$. (Here, we view dual
  quaternions as real vectors of dimension eight.) The empty product
  is included, its value is $1$.
\end{defn}

\begin{defn}
  \label{def:4}
  The dimension of the coupling space $L_{i,i+1,\ldots,j}$ will be
  called the \emph{coupling dimension}. We denote it by
  $l_{i,i+1,\ldots,j} = \dim L_{i,i+1,\ldots,j}$.
\end{defn}

Note that Definitions~\ref{def:3}--\ref{def:5} also make sense if we
arrange the consecutive joints in decreasing order with respect to our
chosen linkage representation $(h_1,h_2,\ldots,h_n)$. With this in
mind, we also write $L_{i,i+1,\ldots,j} = L_{i,i-1,\ldots,j}$ or
$l_{i,i+1,\ldots,j} = l_{i,i-1,\ldots,j}$. This consideration also
applies to

\begin{defn}
  \label{def:5}
  For a sequence $h_i,h_{i+1},\dots,h_j$ of consecutive joints, we
  define the \emph{coupling variety $X_{i,i+1,\dots,j}\subset\P^7$} as
  the set of all products $(t_i-h_i)\dots(t_j-h_j)$ with $t_k\in\P^1$
  for $k=i,\dots,j$ or, more precisely, the set of all equivalence
  classes of these products in the projective space.
\end{defn}

The coupling variety is a subset of the projectivization of the
coupling space. The relation between the coupling curve and the
coupling variety is described by the ``coupling equality''
$C_{i,j}=X_{i+1,\dots,j}\cap X_{i-1,\dots,-n+j}$.

We also recall the nomenclature of \cite{HSS,HSS2}: Two rotation
quaternions with the same axes are called \emph{compatible.} Moreover,
two or more lines are called \emph{concurrent} if they are all
parallel or intersect in a common point.

We now prove a theorem that relates the introduced concepts to the
axis geometry of the linkage. We will use it later to show that bonds
have a geometric meaning but it has aspects, which are interesting in
its own right, for example
Theorem~\ref{thm:coup}.d.

\begin{lem}
  \label{lem:field}
  The triple $(L_1,+,\cdot)$ is a field and isomorphic to~$\C$.
\end{lem}
\begin{proof}
  The set $L_1 = \{a+bh_1 \mid a,b \in \R\}$ is closed under addition. Since quaternions in $L_1$
  describe rotations about one fixed axis it is also closed under
  multiplication and inversion. This already implies that $L_1$ is a
  field. Because of $h_1^2 = -1$, $L_1$ is isomorphic to~$\C$.
\end{proof}

\begin{thm} \label{thm:coup}
  If $h_1,h_2,\ldots,h_n$ are rotation quaternions such that $h_i$ and
  $h_{i+1}$ are not compatible for $i=1,\ldots,n-1$, the following
  statements hold true:
  \begin{enumerate}[label=\alph*)]
  \item All coupling dimensions $l_{1,\ldots,i}$ with $1 \le i \le n$
    are even.
  \item The equation $l_{1,2}=4$ always holds. Moreover, $L_{1,2}\subset
    S$ if and only if the axes of $h_1$ and $h_2$ are concurrent.
  \item If $\dim L_{1,2,3} = 4$, then the axes of $h_1$, $h_2$, $h_3$ are
    concurrent.
  \item If $\dim L_{1,2,3} = 6$, then the axes of $h_1$, $h_2$, $h_3$ satisfy
    the Bennett conditions: the normal feet of $h_1$ and $h_3$ on
    $h_2$ coincide and the normal distances $d_{i,i+1}$ and angles
    $\alpha_{i,i+1}$ between consecutive axes are related by
    $d_{12}/\sin\alpha_{12} = d_{23}/\sin\alpha_{23}$.
  \end{enumerate}
\end{thm}

\begin{proof}
  a) The coupling space $L_{1,\dots,i}$ is closed under multiplication
  with $L_1$ from the left. Hence $L_{1,\dots,i}$ is a vector space
  over the field $L_1$. By Lemma~\ref{lem:field}, $L_1$ is isomorphic
  to $\C$. Hence, the real dimension of $L_{1,\dots,i}$ is even.\par
  b) is well-known \cite[Section~11.2.1]{selig05}.\par
  c) Since both $L_{1,2,3}$ and $L_{1,2}$ have dimension~4, the two
  vectorspaces are equal and $h_3\in L_{1,2}$. Our proof is by
  contradiction. Suppose that $h_1$ and $h_2$ are not concurrent. Then
  the only rotations in the projectivizations are compatible with
  $h_1$ or $h_2$. By assumption, $h_3$ is not compatible with $h_2$,
  hence $h_3=\pm h_1$. Then $L_{1,2,3}$ is closed under multiplication
  by $h_1=\pm h_3$ from the left and from the right. On the other
  hand, no proper subalgebra of $\D\H$ can contain two skew lines,
  hence $h_1$ and $h_2$ are concurrent.
  \par
  d) If $h_3 \in L_{1,2}$, it is either compatible with $h_1$ or $h_2$.
  The latter is excluded by assumption, the former satisfies the
  Bennett conditions. Hence, we can assume $h_3 \notin L_{1,2}$ and the
  vectors $1$, $h_1$, $h_2$, $h_3$, $h_1h_2$ are linearly independent.
  As an $L_1$-vectorspace, $L_{1,2,3}$ is generated by
  $1,h_2,h_3,h_2h_3$. Assume that these vectors form a basis of
  $L_{1,2,3}$. Then $w + xh_2 + yh_3 + zh_2h_3 = 0$ with $w,x,y,z \in
  L_1$ would imply $w = x = y = z = 0$ so that $l_{1,2,3} = 8$. This
  contradicts our assumption. Hence, there is a non-trivial linear
  relation
  \begin{equation}
    \label{eq:5}
    x + yh_2 + zh_3 = h_2h_3
  \end{equation}
  with unique $x,y,z\in L_1$. By multiplying \eqref{eq:5} from the
  right with $h_3$, we obtain $xh_3+yh_2h_3-z=-h_2$. Comparing
  coefficients with \eqref{eq:5} then yields $y^2=-1$, $z=xy$, and
  $x=zy$. We may assume, possibly after replacing $h_1$ by $-h_1$,
  that $y=-h_1$. Then we can also write $x=a+bh_1$ and $z=b-ah_1$ for
  some $a,b\in\R$. If $a = 0$, \eqref{eq:5} becomes
  $(h_2-b)h_3=h_1(b-h_2)$, there is a rotation around $h_2$ that
  transforms $h_1$ to $h_3$ and the claim follows. If $a \neq 0$, we
  set $h_2':=a^{-1}(h_2-b)$ (another rotation about the same axis) and
  find
  \begin{equation*}
    \begin{gathered}
      a(h_1h_2'+h_1h_3+h_2'h_3) =\\
      h_1h_2-bh_1+ah_1h_3+h_2h_3-bh_3 =\\
      h_1h_2-bh_1+ah_1h_3+a+bh_1-h_1h_2+bh_3-ah_1h_3-bh_3 = a.
    \end{gathered}
  \end{equation*}
  It follows that $h_4:=-h_1-h_2'-h_3$ fulfills the two equations
  $h_1+h_2'=\bar{h_4}+\bar{h_3}$, $h_1h_2'=\bar{h_4}\,\bar{h_3}$.
  Hence, the closure equation $(t-h_1)(t-h_2)(t-h_3)(t-h_4) \in \R$ of
  Bennett's mechanism is fulfilled (see \cite{HSS,HSS2}).
\end{proof}

\section{Bonds}
\label{sec:bonds}

In this article's central section we define bonds and introduce the
bond structure (local distance and local joint length). We show how
the bond structure can be used to compute the degree of coupling
curves and derive some algebraic implications of the theory. Towards
the end of this section, we introduce the connection numbers
associated to bonds and use them for drawing bond diagrams. From now
on, we consider closed revolute chains with incompatible consecutive
axes only.

\subsection{Definition of bonds}

Consider a closed chain $L = (h_1,\ldots,h_n)$ of mobility one with
configuration curve $K$. By $K_\C$ we denote its Zariski closure, the
set of all points in $(\P^1_\C)^n$ which satisfy all algebraic
equations that are also satisfied by all points of~$K$. Now we set
\begin{equation}
  \label{eq:6}
  B := \{(t_1,\ldots,t_n) \in K_\C \mid
  (t_1-h_1)(t_2-h_2)\cdots(t_n-h_n) = 0\}.
\end{equation}

\begin{prop}
  We have $\dim(B)=\dim(K)-1$.
\end{prop}

\begin{proof}
  The ideal of $B$ is generated by the ideal of $K$ and one additional
  equation, the primal scalar part of
  $(t_1-h_1)(t_2-h_2)\cdots(t_n-h_n)$. Hence $B$ is a hypersurface in
  $K_\C$ and it follows that $\dim(B)\ge\dim(K)-1$. If
  $\dim(B)=\dim(K)$, then there would be a component of $K_\C$ that
  would entirely lie in $B$. But this is impossible because $B$ has no
  real points and $K$ is entirely real.
\end{proof}

The set $B$ is a finite set of conjugate complex points on the
configuration curve's Zariski closure. These points are special in the
sense that they, by defining condition \eqref{eq:6} of bonds, do not
correspond to a valid linkage configuration.

In \cite{HSS3}, we simply defined a bond as a point of $B$. But we
also remarked that this is only valid in ``typical'' cases. Here, we
adopt a more general point of view. It is conceivable that $K_\C$ is
singular at a point of $B$ so that more than one bond lies over this
point. In order to overcome this technical difficulty, we consider the
\emph{normalization $\normalization(K)$} instead of $K_\C$ (see
\cite{Shafarevich:74} Chapter II.5). The normalization
$\normalization(K)$ is a singularity-free curve that serves as
parameter range for $K_\C$. In other words, there exists a surjection
$\nu\colon \normalization(K) \to K_\C$, the \emph{normalization map.}

\begin{defn}
  \label{def:bond}
  Let $\normalization(K)$ be the normalization of the algebraic curve
  $K_\C$, with normalization map $\nu\colon \normalization(K)\to
  K_\C$. A point $\beta \in \normalization(K)$ is called a \emph{bond}
  if $\nu(\beta) \in B$.
\end{defn}

We mention that it is usually possible to think of a bond as a point
$\beta \in B$. The concept of normalization is only needed if $K_\C$ is
singular at $\beta$ -- a situation we will not encounter in this
paper.

In the following, we denote the standard basis of the dual quaternions
$\D\H$ by $(1,\qi,\qj,\qk$, $\eps,\eps\qi,\eps\qj,\eps\qk)$ and the
imaginary unit in the field of complex numbers $\C$ by $\ci$. Often,
complex numbers are embedded into the quaternions by identifying $\ci$
with $\qi$. In this paper, we do not do this. \emph{It is crucial to
  distinguish between the imaginary unit $\ci$ and the quaternion
  $\qi$.} We will, for example, encounter expressions like $\ci -
\qi$. This is a quaternion with complex coefficients and different
from zero.

As a first example, we compute the bonds of a Bennett linkage. (The
source code for computing the following examples can be found %
on the accompanying web-site %
\url{http://geometrie.uibk.ac.at/schroecker/bonds/}.)

\begin{example}[Bennett linkage]
  \label{ex:1}
  Consider the linkage $L = (h_1,h_2,h_3,h_4)$ with
  \begin{equation*}
    \begin{aligned}
    h_1 &= \qi,\\
    h_2 &= 9\eps\qi + \qj - 9\eps\qk,\\
    h_3 &= -(\tfrac{1}{3}+4\eps)\qi - (\tfrac{2}{3} - 4\eps)\qj + (\tfrac{2}{3}+2\eps)\qk,\\
    h_4 &= (\tfrac{2}{3}+5\eps)\qi + (\tfrac{1}{3}+4\eps)\qj + (\tfrac{2}{3}-7\eps)\qk.
    \end{aligned}
  \end{equation*}
  From the closure condition $(t_1-h_1)(t_2-h_2)(t_3-h_3)(t_4-h_4) \in
  \R$ we can compute the parametrized representation
  \begin{equation}
    \label{eq:7}
    t_1 = t-1,\
    t_2 = t,\
    t_3 = t-1,\
    t_4 = -t,
    \quad t \in \P^1
  \end{equation}
  of the configuration curve. It is, indeed, of dimension one and $L$
  is a flexible closed 4R chain. It is well-known that any such
  linkage is either planar, spherical or a Bennett linkage
  \cite[Chapter~10, Section~5]{hunt78}. We have
  \begin{equation}
    \label{eq:8}
    \begin{gathered}
      (t-1-h_1)(t-h_2)(t-1-h_3)(-t-h_4)=-(t^2+1)(t^2-2t+2).
    \end{gathered}
  \end{equation}
  The bonds can be computed by solving
  $(t_1-h_1)(t_2-h_2)(t_3-h_3)(t_4-h_4) = 0$. This means, that we have
  to find the zeros of \eqref{eq:8}. They are $t = \pm \ci$ and $t = 1
  \pm \ci$ so that the bond set $B$ consists of the four points
  \begin{equation}
    \label{eq:9}
    (t_1,t_2,t_3,t_4) = (\pm\ci, 1\pm\ci, \pm\ci, -1\mp\ci),\quad
    (t_1,t_2,t_3,t_4) = (-1\pm\ci, \pm\ci, -1\pm\ci, \mp\ci).
  \end{equation}
  We observe that every bond of \eqref{eq:9} has two entries equal to
  $\ci$ or $-\ci$. As next theorem shows, this is no coincidence
  but a typical property of bonds.
  \hfill\eoex
\end{example}

\begin{thm}
  \label{thm:i}
  For a bond $\beta \in \nu^{-1}(t_1,\ldots,t_n)$ there exist indices
  $i,j \in [n]$, $i < j$, such that $t_i^2 + 1 = t_j^2 + 1 = 0$.
\end{thm}
\begin{proof}
  Observe at first that for any $k \in [n]$ the equality
  \begin{equation*}
    \norm(t_k-h_k) = (t_k-h_k)(\bar{t_k-h_k}) = (t_k-h_k)(t_k+h_k) = t_k^2 + 1
  \end{equation*}
  holds. ($\norm(h)=h\bar{h}$ is the norm of a dual quaternion.)
  Taking the norm on both side of the defining condition \eqref{eq:6}
  of bonds, we obtain
  \begin{equation}
    \label{eq:10}
    0 = \prod_{k=1}^n (t_k-h_k) \prod_{k=1}^n (t_{n+1-k}+h_{n+1-k})
      = \prod_{k=1}^n (t_k^2+1).
  \end{equation}
  We conclude that $t_i^2 + 1 = 0$ for at least one index $i \in [n]$
  and we assume that $i$ is the minimal index with this property. In
  order to show existence of a second index $j \in [n]$, $i < j$ with
  $t_j^2 + 1 = 0$, we successively multiply the bond equation
  \eqref{eq:6} with $t_n+h_n$, \ldots, $t_{i+1}+h_{i+1}$ from the right
  and with $t_1+h_1$, \ldots, $t_{i-1}+h_{i-1}$ from the left. The
  result is
  \begin{gather*}
    0 =
    \prod_{k=1}^{i-1} (t_k+h_k) \prod_{k=1}^n (t_k-h_k) \prod_{k=1}^{n-i} (t_{n+1-k}+h_{n+1-k}) =
    (t_i-h_i) \prod_{k \neq i}(t_k^2+1).
  \end{gather*}
  Now the claim follows because $t_i-h_i$ never vanishes.
\end{proof}

\begin{defn}
  \label{def:typical}
  We call a bond $\beta = \nu^{-1}(t_1,\ldots,t_n)$ \emph{typical} if
  there are precisely two indices $i,j\in[n]$, $i < j$ such that
  $t_i^2+1=t_j^2+1=0$.
\end{defn}

Theorem~\ref{thm:i} is important for two reasons. First of all, it
gives us necessary conditions that are useful for the actual
computation of typical bonds. Secondly, it is a further manifestation
of the mentioned discrete properties of bonds: For a typical bond
$\beta$, the two links $h_i$, $h_j$ with $t_i^2 + 1 = t_j^2 + 1 = 0$
play a special role. We say that the bond ``connects'' $h_i$ and
$h_j$. However, this concept requires a more refined elaboration as we
also have to take into account non-typical cases and higher connection
multiplicities. For this reasons, the precise definition of a
connection number between two joints is necessary. This needs more
preparation work and will be deferred until
Section~\ref{sec:bond-diagrams}.

\begin{cor}
  \label{cor:prod}
  For a typical bond $\beta = \nu^{-1}(t_1,\ldots,t_n)$ with
  $t_i^2+1=t_j^2+1=0$ and $i < j$, the equalities
  \begin{equation}
    \label{eq:11}
    F_{i-1,j}(t_1,\ldots,t_n) = F_{j-1,n+i}(t_1,\ldots,t_n) = 0
  \end{equation}
  hold.
\end{cor}

\begin{proof}
  Once more, we consider the bond equation \eqref{eq:6}. We multiply
  it from the left with $t_1+h_1,\ldots,t_{i-1}+h_{i-1}$ and from the
  right with $t_{j+1}+h_{j+1},\ldots,t_n+h_n$ to obtain
  \begin{gather*}
    0 = \prod_{k=1}^{i-1} (t_k+h_k) \prod_{k=1}^n (t_k-h_k) \prod_{k=j+1}^n (t_k+h_k)
      = \prod_{k \notin \{i,\ldots,j\}} (t_k^2+1) \prod_{k=i}^j (t_k-h_k).
  \end{gather*}
  Because the first product on the right is different from zero, the
  second vanishes. The second equality can be seen similarly.
\end{proof}

The reader is invited to verify Corollary~\ref{cor:prod} with the data
of Example~\ref{ex:1}.

Before proceeding with our study of bonds, we present two further
examples (spherical and planar four-bar linkage) that illustrate
special situations that can occur: Different bonds may have the same
indices $i < j \in [n]$ such that $t_i^2+1=t_j^2+1=0$ and, for a given
bond, there might exist more than two indices $i < j \in [n]$ with
this property.

\begin{example}[Spherical four-bar linkage]
  \label{ex:2}
  We consider the spherical four-bar linkage $L = (h_1,h_2,h_3,h_4)$
  given by
  \begin{equation*}
    h_1 = \qi,\quad
    h_2 = \qj,\quad
    h_3 = \qk,\quad
    h_4 = \frac{3}{5}\qi + \frac{4}{5}\qj.
  \end{equation*}
  The configuration curve admits the parametrization
  \begin{equation}
    \label{eq:12}
    \begin{gathered}
      t_1 = \frac{5-5t^2 + w}{6t},\quad
      t_2 = \frac{-5t^2-5 + w}{8t},\quad
      t_3 = \frac{25t^2-7 - 5w}{24},\quad
      t_4 = t
    \end{gathered}
  \end{equation}
  where $w = \pm\sqrt{25t^4-14t^2+25}$. The bonds are
  \begin{gather*}
    (\mp 3 \ci, \mp \ci, -3, \pm\ci),\quad
    (\mp\tfrac{1}{3}\ci, \pm\ci, \tfrac{1}{3}, \pm\ci),\quad
    (\mp\ci, -1, \pm\ci, \tfrac{4}{5} \pm \tfrac{3}{5} \ci),\quad
    (\mp\ci, 1, \mp\ci, -\tfrac{4}{5} \pm \tfrac{3}{5} \ci).
  \end{gather*}
  Thus, we have two pairs of conjugate complex bonds with
  $t_1^2+1=t_3^2+1=0$ and two pairs of conjugate complex bonds with
  $t_2^2+1=t_4^2+1=0$.
  \hfill\eoex
\end{example}

\begin{example}[Planar four-bar linkage]
  \label{ex:3}
  The configuration curve of the planar four-bar linkage given by
  \begin{equation*}
    h_1 = \eps\qi + \qk,\quad
    h_2 = \eps\qj + \qk,\quad
    h_3 = \qk,\quad
    h_4 = \eps\qi + 2\eps\qj + \qk
  \end{equation*}
  can be parametrized by
  \begin{equation}
    \label{eq:13}
    \begin{gathered}
      t_1 = \frac{-t^2 + 2t + 5 - w}{2(t+3)},\quad
      t_2 = \frac{t^2 + 1 + w}{4(2-t)},\quad
      t_3 = \frac{t^2-4t+1 + w}{4(t-1)},\quad
      t_4 = t
    \end{gathered}
  \end{equation}
  where $w = \pm\sqrt{t^4-8t^3+2t^2+56t-47}$. The bonds are
  \begin{gather*}
    (\pm\ci, -2 \pm \ci, \mp\ci, 4 \mp \ci),\quad
    (2 \pm \ci, \mp\ci, -1 \pm 2\ci, \pm\ci),\quad
    (\pm\ci, \mp\ci, \pm\ci, \mp\ci).
  \end{gather*}
  The special thing here is the existence of two non-typical bonds.
  For them, Corollary~\ref{cor:prod} cannot be applied. Nonetheless,
  we observe that
  \begin{equation*}
    \begin{aligned}
      (t_1 - h_1)(t_2 - h_2)(t_3 - h_3) &= (t_2 - h_2)(t_3 - h_3)(t_4 - h_4) = \\
      (t_3 - h_3)(t_4 - h_4)(t_1 - h_1) &= (t_4 - h_4)(t_1 - h_1)(t_2 - h_2) = 0
    \end{aligned}
  \end{equation*}
  holds for $(t_1,t_2,t_3,t_4) = (\pm\ci, \mp\ci, \pm\ci, \mp\ci)$.
  \hfill\eoex
\end{example}

\begin{rem}
  So far, we silently ignored the possibility of a bond with
  coordinate $\infty$. Actually, no such bonds occur in our examples
  but this has to be checked carefully. Linkages with ``bonds at
  infinity'' do exist.
\end{rem}

Example~\ref{ex:3} is an indication that the vanishing of coupling
maps as stated in Corollary~\ref{cor:prod} for typical bonds, is a
more relevant property than existence of indices $i<j\in[n]$ with
$t_i^2+1=t_j^2+1=0$, as stated by Theorem~\ref{thm:i}. In the
following, we will elaborate this concept in more detail.

\subsection{Local distances and joint lengths}
\label{sec:local-distances}

Now we are going to define local distances and joint lengths of a
linkage. These are algebraic notions related to a single bond. In
Section~\ref{sec:distances} we will define (non-local) distances and
joint lengths as sum over all local distances and joint lengths,
respectively.

The definition of local distances requires the concept of the
vanishing order of a function $f\colon K_\C \to \P^7$ at a bond
$\beta$. Consider an arbitrary homogeneous quadratic form $F\colon
\R^8 \to \C$. The image $F(x)$ of a vector $x \in \R^8$ is obtained by
plugging the coordinates of $x$ into a homogeneous quadratic
polynomial. The function $F$ is not necessarily well-defined on $\P^7$
but the vanishing order $\ord_\beta(F(f))$ of $F \circ f$ at $\beta$
\cite[see][p.~96]{danilov98} is well-defined. Note that
$\ord_\beta(F(f))=0$ if $F(f(\beta)) \neq 0$. In this article, we will
use the homogeneous quadratic form $Q$ which maps $x =
(x_0,\ldots,x_7) \in \R^8$ to $Q(x) = x_0^2+x_1^2+x_2^2+x_3^2$ (the
primal part of $\norm(x)$).

\begin{defn}
  \label{def:local-distance}
  For a bond $\beta \in \normalization(K)$ and a pair $(i,j)$ of
  links, the \emph{local distance} is defined as $ d_\beta(i,j) :=
  \frac{1}{2}\ord_\beta Q(f_{i,j})$ where $f_{i,j}$ is the coupling
  map of Definition~\ref{def:coupler}. The \emph{local distance matrix
    $D_\beta$} is the matrix with entries $d_\beta(i,j)$. The
  \emph{local joint length} is defined as $b_\beta(i) :=
  d_\beta(i-1,i) = \frac{1}{2}\ord_\beta Q(t_i-h_i)$.
\end{defn}

\begin{rem}
  Definition~\ref{def:local-distance} relates bond theory to a
  familiar concept of theoretical kinematics. If $d_\beta(i,j)$ is
  positive, the image of the bond $\beta$ under the coupling map
  $f_{i,j}$ is a point $x = (x_0,\ldots,x_7) \in C_{i,j}$ such that
  $x_0^2 + x_1^2 + x_2^2 + x_3^2 = 0$. This equation describes a
  quadratic cone $G$ whose vertex space is the exceptional 3-plane
  $E$. Intersections of the coupling curve with $E$ have been
  considered before (e.g. in \cite[Chapter~11]{selig05} and
  \cite{husty10}), this article shows that it is even more interesting
  to study the intersection points with~$G$.
\end{rem}

\begin{example}
  \label{ex:4}
  The local distance matrices for the Bennett linkage of
  Example~\ref{ex:1} are
  \begin{equation}
    \label{eq:14}
    D_{\beta'} = \frac{1}{2}
    \begin{pmatrix}
      0   & 0   & 1 & 1 \\
      0   & 0   & 1 & 1 \\
      1 & 1 & 0   & 0   \\
      1 & 1 & 0   & 0
    \end{pmatrix},
    \quad
    D_{\beta''} = \frac{1}{2}
    \begin{pmatrix}
       0   & 1 & 1 & 0   \\
       1 & 0   & 0   & 1 \\
       1 & 0   & 0   & 1 \\
       0   & 1 & 1 & 0
    \end{pmatrix}
  \end{equation}
  with $\beta' = (\pm\ci, 1\pm\ci, \pm\ci -1\mp\ci)$ and $\beta'' =
  (-1\pm\ci, \pm\ci, -1\pm\ci, \mp\ci)$. For their computation, we can
  use the parametrized representation \eqref{eq:7}. We find, for
  example,
  \begin{multline*}
    F_{1,3}(t) = (t-h_2)(t-1-h_3) = t^2-t+5\eps+\tfrac{2}{3} \\
    + (t(\tfrac{1}{3}-5\eps) + \tfrac{2}{3}+ 5\eps)\qi +
    (-t(\tfrac{1}{3}+4\eps) + 1 - 3\eps)\qj \\
    +(-t(\tfrac{2}{3}+7\eps) + \tfrac{1}{3} - 11\eps)\qk.
  \end{multline*}
  The bond $\beta = (-1-\ci,-\ci,-1-\ci,\ci)$ belongs to the parameter
  value $t = \ci$. Because $F_{1,3}(\ci) \neq 0$, we can compute the
  local distance as half the vanishing order of
  \begin{equation*}
    Q(F_{1,3}(t)) = (1+t^2)(t^2-2t+2)
  \end{equation*}
  at $t = \ci$, that is, $d_\beta(1,3) = \frac{1}{2}$.
  \hfill\eoex
\end{example}

\begin{example}
  \label{ex:5}
  The local distances for the bonds of Example~\ref{ex:3} (planar
  four-bar linkage) are
  \begin{equation}
    \label{eq:15}
    \begin{gathered}
      D_{\beta'} = \frac{1}{2}
      \begin{pmatrix}
        0 & 1 & 1 & 0 \\
        1 & 0 & 0 & 1 \\
        1 & 0 & 0 & 1 \\
        0 & 1 & 1 & 0
      \end{pmatrix},
      \quad D_{\beta''} = \frac{1}{2}
      \begin{pmatrix}
        0 & 0 & 1 & 1 \\
        0 & 0 & 1 & 1 \\
        1 & 1 & 0 & 0 \\
        1 & 1 & 0 & 0
      \end{pmatrix},
      \quad D_{\beta'''} = \frac{1}{2}
      \begin{pmatrix}
        0 & 1 & 2 & 1 \\
        1 & 0 & 1 & 2 \\
        2 & 1 & 0 & 1 \\
        1 & 2 & 1 & 0
      \end{pmatrix},
    \end{gathered}
  \end{equation}
  where $\beta' = (2\pm\ci, \mp\ci, -1\pm2\ci, \pm\ci)$, $\beta'' =
  (\pm\ci, -2\pm\ci, \mp\ci, 4\mp\ci)$, and $\beta''' =
  (\pm\ci,\mp\ci,\pm\ci,\mp\ci)$. We introduce a new aspect and
  discuss the computation of $d_\beta(1,4)$ for the non-typical bond
  $\beta=(-\ci,\ci,-\ci,\ci)$. It corresponds to $t = \ci$ and the
  positive square root $w$ in the parametrized equation \eqref{eq:13}.
  Because of $f_{1,4} = f^{-1}_{4,1}$ and because inversion is, up to
  scalar multiplication, equal to conjugation, we clearly have
  $d_\beta(1,4) = d_\beta(4,1)$. The latter local distance is easily
  computed to be $d_\beta(4,1) = \frac{1}{2}\ord_\beta Q(t_1-h_1) =
  \frac{1}{2}$.

  However, when we insert the parametrized equation \eqref{eq:13} into
  the product $F_{1,4} = (t_2-h_2)(t_3-h_3)(t_4-h_4)$, we see that it
  vanishes at $t = \ci$. Thus, the parametrization \eqref{eq:13} does
  not give a well-defined map into $\P^7$ at the bond $\beta$ and we
  have to compute the local distance as
  \begin{equation}
    \label{eq:16}
    d_\beta(1,4) = \frac{1}{2}\ord_\beta Q(F_{1,4}) - \min\ord_\beta(F_{1,4}) = \frac{1}{2}
  \end{equation}
  where $\min\ord_\beta(F_{1,4})$ denotes the minimal vanishing order
  of the coordinates of $F_{1,4}$ at $\beta$. This vanishing order
  enters with multiplicity two in the norm, so that the factor
  $\frac{1}{2}$ can be omitted. The actual evaluation of \eqref{eq:16}
  by means of the parametrized equation \eqref{eq:13} poses no
  problems.
  \hfill\eoex
\end{example}

Below, we present an alternative method for computing local distances
using the products $F_{i,j}$ as functions from $\normalization(K)$ to
$\D\H$. As a consequence, we are able to derive a couple of
interesting properties of the local distance function
(Theorem~\ref{thm:local}).

\begin{lem}
  \label{lem:dbeta}
  The local distance $d_\beta(i,j)$ can be computed as
  \begin{equation}
    \label{eq:17}
    d_\beta(i,j)=\sum_{k=i+1}^j b_\beta(k) - v_\beta(i,j)
  \end{equation}
  where $v_\beta(i,j) = \min\ord_\beta(F_{i,j})$ is the minimal
  vanishing order of the coordinates of $F_{i,j}$ at~$\beta$.
\end{lem}

\begin{proof}
  If $v_\beta(i,j) = 0$, then the product of $F_{k-1,k}$ for $k
  = i+1,\ldots,j$ does not vanish at $\beta$, and gives $F_{i,j}$. The
  primal part of the norm is multiplicative, and this implies the
  equation. In the general case, the product is equal to
  $u^mf_{i,j}$ for some local parameter $u$ at $\beta$ and $m =
  v_\beta(i,j)$, and this gives precisely the correction stated in the
  equation.
\end{proof}

\begin{example}
  \label{ex:6}
  We continue Example~\ref{ex:4} and compute $d_\beta(1,3)$ at $\beta
  = (-1-\ci,-\ci,-1-\ci,\ci)$ also by means of Lemma~\ref{lem:dbeta}.
  From the matrices in \eqref{eq:14} we read off:
  \begin{equation*}
    b_\beta(2) = d_\beta(1,2) = \frac{1}{2},\quad
    b_\beta(3) = d_\beta(2,3) = 0.
  \end{equation*}
It is easy to verify that
  $F_{1,3}(\ci) \neq 0$ so that $\min\ord_\beta(f_{1,3}) = 0$ and
  \begin{equation*}
    d_\beta(1,3) = b_\beta(2) + b_\beta(3) - v_\beta(1,3) = \frac{1}{2} + 0 - 0 = \frac{1}{2},
  \end{equation*}
  as expected.
  \hfill\eoex
\end{example}

\begin{example}
  \label{ex:7}
  We continue with Example~\ref{ex:5} and compute $d_\beta(1,j)$ at
  $\beta = (\ci, -\ci, \ci, -\ci)$ also by means of
  Lemma~\ref{lem:dbeta}. From Equation~\eqref{eq:15} we see that
  \begin{equation*}
    b_\beta(2) = b_\beta(3) = b_\beta(4) = \frac{1}{2}.
  \end{equation*}
  For computing the local distances, we also need the vanishing orders
  $v_\beta(1,j)$. Since $\beta$ belongs to the parameter value $t =
  \ci$ in the parametrization \eqref{eq:13} with positive sign of the
  square root, we have to compute the minimum vanishing order of the
  coordinates of $F_{1,j}(t)$ at $t = \ci$. We have
  \begin{equation*}
    F_{1,2}(\ci) \neq 0,\quad
    F_{1,3}(\ci) \neq 0,\quad
    F_{1,4}(\ci) = 0,\quad
    \frac{\mathrm{d}}{\mathrm{d}t}F_{1,4}(\ci) \neq 0.
  \end{equation*}
  Hence $v_\beta(1,2) = v_\beta(1,3) = 0$, $v_\beta(1,4) = 1$ and
  \begin{gather*}
    d_\beta(1,2) = \frac{1}{2} - 0 = \frac{1}{2},\quad
    d_\beta(1,3) = \frac{1}{2} + \frac{1}{2} - 0 = 1,\quad
    d_\beta(1,4) = \frac{1}{2} + \frac{1}{2} + \frac{1}{2} - 1 = \frac{1}{2},
  \end{gather*}
  as expected.
  \hfill\eoex
\end{example}

From these examples, some properties of local bonds are fairly
obvious. We state and prove them formally in

\begin{thm}
  \label{thm:local}
  For each bond $\beta$, the local distance $d_\beta$ has the
  following properties:
  \begin{enumerate}[label=\alph*)]
  \item The local distance is a pseudometric on $[n]$: For all $i \le
    j \le k \in [n]$ we have
      \begin{itemize}
      \item $d_\beta(i,i) = 0$,
      \item $d_\beta(i,j) = d_\beta(j,i)$,
      \item $d_\beta(i,k) \le d_\beta(i,j) + d_\beta(j,k)$ (triangle
        inequality).
      \end{itemize}
  \item $d_\beta(i,j)+d_\beta(j,k)+d_\beta(i,k) \in \N$
  \item $d_\beta(i-1,i+1)=b_\beta(i)+b_\beta(i+1)$
  \end{enumerate}
\end{thm}

\begin{proof}
  a) The first item is true because $d_\beta(i,i)$ is the vanishing
  order of the empty product whose value is defined to be $1$. The
  second item is true because $d_\beta(j,i)$ is the vanishing order of
  $Q(f_{j,i})$ at $\beta$. It equals $Q(f_{i,j})$ because $f_{j,i}$ is
  the conjugate of $f_{i,j}$. In order to prove the triangle
  inequality, we observe that $v_\beta(i,j) + v_\beta(j,k) \le
  v_\beta(i,k)$ because the formal product for computing the
  right-hand side can be factored into the formal products for
  computing the left-hand side. Thus, by Equation~\eqref{eq:17}, we
  have
  \begin{multline*}
    d_\beta(i,j) + d_\beta(j,k)
    = \sum_{l=i+1}^j b_\beta(l) - v_\beta(i,j) + \sum_{l=j+1}^k b_\beta(l) - v_\beta(j,k) \\
    = \sum_{l=i+1}^k b_\beta(l) - v_\beta(i,j) - v_\beta(j,k)
    \ge \sum_{l=i+1}^k b_\beta(l) - v_\beta(i,k) = d_\beta(i,k).
  \end{multline*}
  \par
  b) By Equation \eqref{eq:17} we have
  \begin{multline*}
    d_\beta(i,j) + d_\beta(j,k) + d_\beta(i,k) =
    2\sum_{l=i+1}^k b_\beta(l) - v_\beta(i,j) - v_\beta(j,k) - v_\beta(i,k).
  \end{multline*}
  The right-hand side is a sum of integers and the left-hand side is
  non-negative.\par
  c) is equivalent to $v_\beta(i-1,i+1) = 0$, that is, the product
  $(t_i-h_i)(t_{i+1}-h_{i+1})$ does not vanish at $\beta$. Expanding
  this product and assuming, to the contrary, that it does vanish, 
  we get a nontrivial relation with complex coefficients
  between the vectors $1,h_i,h_{i+1},h_ih_{i+1}$. Its real or complex
  part is a nontrivial relation with real coefficients. Under our
  general assumption that two consecutive revolute axes are never
  identical, this contradicts Theorem~\ref{thm:coup}.b).
\end{proof}

\subsection{Distances and joint lengths}
\label{sec:distances}

Now we introduce (non-local) distances and joint lengths and relate
them to local distances and joint lengths.

\begin{defn}
  \label{def:distance}
  The \emph{distance $d(i,j)$} is defined as
  $d(i,j):=\deg(C_{i,j})\deg(f_{i,j})$, where $\deg(C_{i,j})$ is the
  degree of the coupling curve as a projective curve in $\P^7$ and
  $\deg(f_{i,j})$ is the algebraic degree of the coupling map
  $f_{i,j}\colon K\to C_{i,j}$, that is, the cardinality of a generic
  pre-image when we consider also complex points of $K$. Moreover, we
  write $b(i):=d(i-1,i)$ for $i=1,\dots,n$, $d(0,1)=d(n,1)$ and call
  the numbers $b(1),\dots,b(n)$ the \emph{joint lengths.}
\end{defn}

The definition of $d(i,j)$ as geometric degree times multiplicity
suggests to refer to it also as \emph{algebraic degree} of the
coupling curve $C_{i,j}$.

It is a good point to clarify some of our terminology. When we speak
of a coupling curve, we mean the relative motion between two links. In
the Study quadric model of Euclidean displacements this is, indeed, a
curve. To us, the degree of a motion is the degree of the
corresponding curve on the Study quadric. This differs from the notion
of a motion's degree as the degree of a generic trajectory. Twice the
degree of the curve on the Study quadric is an upper bound for the
trajectory degree.

Since the coupling curve $C_{i-1,i}$ is a straight line (corresponding
to the rotation around the axis $h_i$), $\deg(C_{i-1,i}) = 1$ and
$b(i)$ just equals the degree of the map $f_{i-1,i}$. In particular,
if $b(i)=1$, all coupling curves can be parametrized by the revolute
angle at $h_i$ (this angle unambiguously determines the linkage
configuration).

\begin{thm}
  \label{thm:bond}
  \begin{enumerate}[label=\alph*)]
  \item The distance $d$ is the sum of the local distances:
    $d(i,j)=\sum_\beta d_\beta(i,j)$ for all $i,j\in [n]$.
  \item The distance $d$ is a pseudometric on $[n]$.
  \item For $i \le j \le k\in [n]$, $d(i,j) + d(j,k) + d(i,k)$ is a
    positive even integer.
  \item For $i\in [n]$, we have $d(i-1,i+1) = b(i) + b(i+1)$.
  \end{enumerate}
\end{thm}

\begin{proof}
  a) For computing $d(i,j)$, we can take any quadratic form that does
  not vanish on $C_{i,j}$, count the points in $\normalization(K)$
  where this form vanishes (counting means with multiplicities), and
  divide by two. We take $Q$, the primal part of the norm, as
  quadratic form. The points where $Q$ vanishes are bonds, and the
  multiplicity of $\beta$ is $2d_\beta(i,j)$.\par
  b), c) and d) are easy consequences of a) and the corresponding
  statements in Theorem~\ref{thm:local}. For statement c), we also
  need to observe that bonds always come as conjugate pairs, and the
  local distances for conjugate bonds are equal.
\end{proof}

The importance of Theorem~\ref{thm:bond} lies in the fact that it
connects local distances, which are part of the bond structure and
have an algebraic meaning, with distances (or algebraic degrees),
which have a geometric meaning. We collect the distances in the
\emph{distance matrix $D = \sum_\beta D_\beta$.}

\begin{example}
  \label{ex:8}
  The distance matrices for the Bennett linkage example, the planar
  four-bar example, and the Goldberg linkage are
  \begin{equation}
    \label{eq:18}
    \begin{pmatrix}
      0 & 1 & 2 & 1 \\
      1 & 0 & 1 & 2 \\
      2 & 1 & 0 & 1 \\
      1 & 2 & 1 & 0
    \end{pmatrix},
    \quad
    \begin{pmatrix}
      0 & 2 & 4 & 2 \\
      2 & 0 & 2 & 4 \\
      4 & 2 & 0 & 2 \\
      2 & 4 & 2 & 0
    \end{pmatrix},
    \quad\text{and}\quad
    \begin{pmatrix}
      0 & 1 & 2 & 3 & 2 \\
      1 & 0 & 1 & 2 & 3 \\
      2 & 1 & 0 & 1 & 2 \\
      3 & 2 & 1 & 0 & 1 \\
      2 & 3 & 2 & 1 & 0
    \end{pmatrix},
  \end{equation}
  respectively. The first and second matrix are obtained by adding the
  matrices given in Equation~\eqref{eq:14} and Equation~\eqref{eq:15},
  respectively, and multiplying them by two (because the bonds come in
  conjugate complex pairs with identical local joint distances). The
  matrix for the Goldberg linkage was obtained by means of
  Theorem~\ref{thm:degree} from the bond diagram in
  Figure~\ref{fig:bond-diagrams}.c).

  In the Bennett case, neighboring links have a relative motion of
  degree one (a rotation about their common axes) and opposite links
  have a relative motion of degree two. In the planar four-bar case,
  the relative motion of neighboring links is still a rotation but a
  generic rotation angle occurs twice. Hence, this motion is of degree
  two. The relative motion of opposite links is of degree four. These
  well-known facts are confirmed by Equations~\eqref{eq:14}
  and~\eqref{eq:15} in conjunction with Definition~\ref{def:distance}.
  \hfill\eoex
\end{example}

\subsection{Connection numbers and bond diagrams}
\label{sec:bond-diagrams}

In this section,
we define the \emph{connection number} for two joints and visualize it
in \emph{bond diagrams.} These are linkage graphs (with vertices
denoting links and edges denoting joints) augmented with additional
connections between certain edges. They serve as a pictorial
representation for part of the information encoded in the linkage's
bond structure. It is possible to directly ``read off'' certain
linkage properties from its bond diagram.

Consider a typical bond $\beta$ with $t_i^2+1=t_j^2+1=0$ for $i < j
\in [n]$. From the linkage graph we remove the edges labeled $h_i$
and $h_j$, thus producing two unconnected chain graphs. Then
$d_\beta(k,l) = 0$ if the vertices labeled $o_k$ and $o_l$ are in the
same component and $d_\beta(k,l) = d_\beta(i-1,i)$ if they are in
different components. We say that the \emph{connection number
  $k_\beta(i,j)$} for this typical bond is equal to $2d_\beta(i-1,i)$
or that \emph{the bond $\beta$ connects $h_i$ and $h_j$ with
  multiplicity $2d_\beta(i-1,i)$.} For the typical bonds in our
examples, we always have $k_\beta(i,j) = 1$. Pictorially, a typical
bond $\beta$ cuts the link diagram into two parts, which are separated
by a fixed distance, generically~$\frac{1}{2}$. The same holds true
for the conjugate bond $\bar{\beta}$. 
Both bonds together
account for the total connection number of $2d_\beta(i-1,i)$.

\begin{defn}
  \label{def:elementary-bond}
  A bond $\beta$ is called \emph{elementary,} if $\sum_{i=1}^n
  b_\beta/2 = 1$.
\end{defn}

Every elementary bond is typical but a typical bond need not be
elementary. The typical bonds in our examples are all elementary. For
a linkage with only elementary bonds, the number $k(i,j)$ of bonds
connecting $h_i$ and $h_j$ equals
\begin{equation}
  \label{eq:19}
  k(i,j) = d(i,j) + d(i-1,j-1) - d(i,j-1) - d(i-1,j).
\end{equation}
Indeed, for an elementary bond $\beta$ we have
\begin{multline*}
  d_\beta(i,j) + d_\beta(i-1,j-1) - d_\beta(i,j-1) - d_\beta(i-1,j) =
  \begin{cases}
    1 & \text{if $\beta$ connects $h_i$ and $h_j$,}\\
    0 & \text{else.}
  \end{cases}
\end{multline*}
By Theorem~\ref{thm:bond}.a), the right-hand side of \eqref{eq:19}
really counts the bonds connecting $h_i$ and $h_j$. These observations
for elementary bonds motivate the following definition for the general
setting.

\begin{defn}
  \label{def:connection-number}
  For a closed linkage $L = (h_1,\ldots,h_n)$ with bond $\beta$ and $i
  < j \in [n]$, the \emph{connection number $k_\beta(i,j)$ at $\beta$}
  is defined as
  \begin{equation}
    \label{eq:20}
    k_\beta(i,j) =
    d_\beta(i,j) + d_\beta(i-1,j-1) - d_\beta(i,j-1) - d_\beta(i-1,j).
  \end{equation}
  We also say that \emph{the bond $\beta$ connects the joints $h_i$
    and $h_j$ with multiplicity~$k_\beta(i,j)$.}
\end{defn}

\begin{lem}
  \label{lem:integer}
  The connection number $k_\beta(i,j)$ is an integer.
\end{lem}
\begin{proof}
  By \eqref{eq:17} and \eqref{eq:20}, we have
  \begin{equation*}
    k_\beta(i,j) =  v_\beta(i,j-1) + v_\beta(i-1,j) - v_\beta(i,j) - v_\beta(i-1,j-1).
  \end{equation*}
  This is a sum of integers.
\end{proof}

We visualize a bond and its connection number by \emph{bond diagrams.}
These are obtained by drawing $k_\beta(i,j)$ connecting lines between
the edges $h_i$ and $h_j$ for each set $\{\beta, \bar{\beta}\}$ of
conjugate complex bonds. 
Since we cannot exclude that $k_\beta(i,j) < 0$, we visualize
negative connection numbers by drawing
the appropriate number of dashed connecting lines (because the
dash resembles a ``minus'' sign). No example in this article
has negative connection number. Actually, the authors do not know
if closed 6R linkages may or may not have bonds with negative
connection numbers. 

\begin{figure*}
  \centering
  \includegraphics{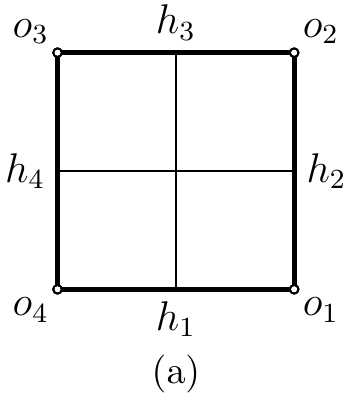}
  \quad
  \includegraphics{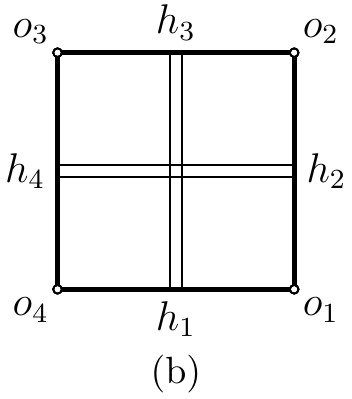}
  \quad
  \includegraphics{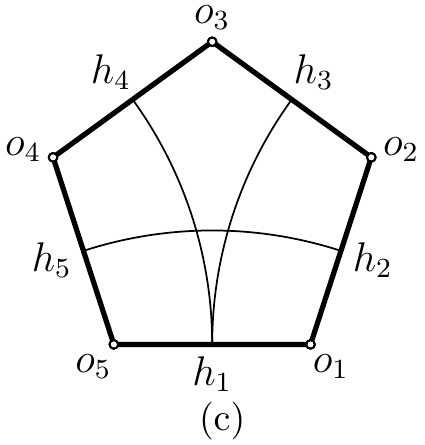}
  \caption{Bond diagrams for the Bennett linkage (a), spherical and planar
    four-bar (b), and Goldberg linkage~(c)}
  \label{fig:bond-diagrams}
\end{figure*}

\begin{example}
  \label{ex:9}
  The bond diagrams for our prototype examples, the Bennett linkage
  and the planar four-bar linkage, are depicted in
  Figure~\ref{fig:bond-diagrams}.a) and b). The elementary bonds with
  $t_i^2+1=t_j^2+1=0$ connect only $h_i$ and $h_j$ with connection
  multiplicity one. The non-typical bond of the planar four-bar
  example connects $h_1$ with $h_3$ and $h_2$ with $h_4$, both with
  connection multiplicity one. Its local distance matrix is sum of the
  elementary bonds' distance matrices. We remark that
  Figure~\ref{fig:bond-diagrams}.b) also gives the bond diagram for
  the spherical four-bar linkage of Example~\ref{ex:2}. Intuitively,
  two elementary bonds of the spherical four-bar coincide in the
  planar four-bar.
  \hfill\eoex
\end{example}

The algebraic degrees of relative coupling motions $C_{i,j}$ are the
entries $d(i,j)$ of the linkage's distance matrix $D = \sum_{\beta}
D_\beta$. These entries can also be read off directly from the bond
diagram which gives us the connection numbers. The following theorem
describes how to do this. It is, essentially, a graphical method to
invert the linear map $\delta$ defined by Equation~\eqref{eq:20}.

\begin{thm}
  \label{thm:degree}
  The algebraic degree of the coupling curve $C_{i,j}$ can be read off
  from the bond diagram as follows: Cut the bond diagram at the
  vertices $o_i$ and $o_j$ to obtain two chain graphs with endpoints
  $o_i$ and $o_j$; the degree of $C_{i,j}$ is the sum of all
  connections that are drawn between these two components (dashed
  connections counted negatively).
\end{thm}

\begin{proof}
  Let $\beta$ be a bond. For any two different links $o_i$ and $o_j$, the number of connections
  between the two subchains obtained by cutting the bond diagram at $o_i$
  and $o_j$ belonging to $\beta$ is equal to 
  \begin{equation}
    \label{eq:21} 
    a_\beta(i,j):=\sum_{r=i+1}^j\sum_{s=j+1}^i k_\beta(r,s).
  \end{equation}
  We identify the space of symmetric $n \times n$ matrices with zero
  diagonal with $\R^N$, $N = {n \choose 2}$ and denote by $K_\beta$
  and $A_\beta$ the matrices with respective entries $k_\beta(i,j)$
  and $a_\beta(i,j)$. Equations~\eqref{eq:20} and \eqref{eq:21} define
  two linear maps $f,g\colon\R^N\to\R^N$, $f(D_\beta) = K_\beta$ and
  $g(K_\beta) = A_\beta$.
  We claim that $f \circ g$ is twice the identity. Indeed, a summand
  $k_\beta(r,s)$ in
  \begin{equation*}
   \sum_{r=i+1}^{j}\sum_{s=j+1}^{i} k_\beta(r,s) +
   \sum_{r=i}^{j-1}\sum_{s=j}^{i-1} k_\beta(r,s) -
   \sum_{r=i+1}^{j-1}\sum_{s=j}^{i} k_\beta(r,s) - 
   \sum_{r=i}^{j}\sum_{s=j+1}^{i-1} k_\beta(r,s) 
  \end{equation*}
  occurs four times with signs $+,+,-,-$ if
  $\{r,s\}\cap\{i,j\}=\varnothing$, twice with signs $+,-$ if
  $\{r,s\}$ and $\{i,j\}$ have one element in common, and twice with
  signs $+,+$ if $\{r,s\}=\{i,j\}$.

  Since $f,g$ are linear maps between finite-dimensional vectorspaces,
  it follows that $g\circ f$ is also twice the identity. Therefore $d_\beta(i,j)=\frac{a_\beta(i,j)}{2}$
  for all pairs $i,j$ such that $i\ne j$. By summing over all bonds, we get the theorem. 
\end{proof}

\begin{example}
  \label{ex:10}
  We illustrate the procedure for computing the distances (or coupling
  curve degrees) in Figure~\ref{fig:degree}. In order to determine the
  degree of the coupling curve $C_{3,5}$, we cut the bond diagram
  along the line through $o_3$ and $o_5$ and count the connections
  between the two chain graphs. There are precisely two of them, one
  connecting $h_1$ with $h_4$ and one connecting $h_2$ with $h_5$.
  Thus, the algebraic degree $d(3,5)$ of $C_{3,5}$ is two. The reader
  is invited to compute the complete data of Equation~\eqref{eq:18} by
  means of the bond-diagrams in Figure~\ref{fig:bond-diagrams}.
  \hfill\eoex
\end{example}

\begin{figure}
  \centering
  \includegraphics{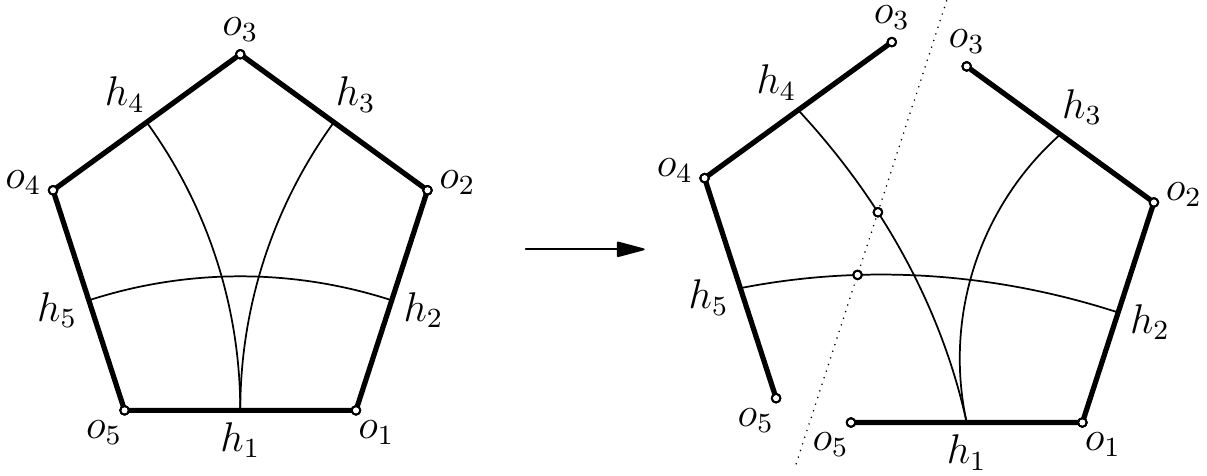}
  \caption{Computing the degree of coupling curves by counting
    connections in the bond-diagram}
  \label{fig:degree}
\end{figure}

In the beginning, when we learned the properties of bonds mostly from
observation, the majority of linkages we studied had only simple
bonds. It occurred to us that these special points on the
configuration curve somehow mysteriously connects two of the $n$
joints, which are not joined by a link. This is the reason for the
name ``bond''. We emphasize that it should not be confused with the
already established concept of a ``kinematic bond''
\cite[Chapter~5]{angeles89}.

\subsection{More properties of bonds}

We briefly state a few additional properties of bonds that follow
immediately from our considerations so far or can easily be shown. We
talk about the bonds of the linkage $L=(h_1,\ldots,h_n)$. Recall also
the introduction of the coupling space dimension $l_{i,i+1,\ldots,j} =
\dim L_{i,i+1,\ldots,j}$ in Definition~\ref{def:4}.

\begin{cor}
  \label{cor:impl}
  If $d(1,4)<d(1,2)+d(2,3)+d(3,4)$, then $l_{2,3,4}\le 6$.
\end{cor}

\begin{proof}
  There must exist at least one bond such that
  $d_\beta(1,4)<d_\beta(1,2)+d_\beta(2,3)+d_\beta(3,4)$. For this
  bond, call it $\beta$, we have $v_\beta(1,4)>0$ by
  Lemma~\ref{lem:dbeta}. Let $t_2$, $t_3$, $t_4$ be the second, third,
  and fourth coordinate of $\beta$, respectively. Since
  $v_\beta(1,4)>0$, the formal product of the corresponding rotations
  vanishes at $\beta$, i.e.\ $(t_2-h_2)(t_3-h_3)(t_4-h_4)=0$.
  Expanding this product, we get a nontrivial relation with complex
  coefficients between the vectors
  $1,h_2,h_3,h_4,h_2h_3,h_2h_4,h_3h_4,h_2h_3h_4$. Either its real or
  its complex part is a nontrivial relation with real coefficients. So
  $l_{2,3,4}$ cannot be eight. By Theorem~\ref{thm:coup}, it cannot be
  seven.
  \hfill\qed
\end{proof}

\begin{cor}
  If a bond $\beta$ connects $h_i$ with $h_{i+2}$, the axes of $h_i$,
  $h_{i+1}$ and $h_{i+2}$ are concurrent or satisfy the Bennett
  conditions, compare Theorem~\ref{thm:coup}.d).
\end{cor}
\begin{proof}
  Without loss of generality, we assume $i = 2$. The connection number
  $k_\beta(i,i+2)$ is positive, that is, $d_\beta(2,4) + d_\beta(1,3)
  - d_\beta(2,3) - d_\beta(1,4) > 0$. Using Theorem~\ref{thm:bond}.d),
  we find
  \begin{equation*}
    \begin{aligned}
      d_\beta(1,4) & < d_\beta(2,4) + d_\beta(1,3) - d_\beta(2,3) \\
                   & = d_\beta(2,3) + d_\beta(3,4) + d_\beta(1,2) + d_\beta(2,3) - d_\beta(2,3) \\
                   & = d_\beta(1,2) + d_\beta(2,3) + d_\beta(3,4).
    \end{aligned}
  \end{equation*}
  By Corollary~\ref{cor:impl}, this implies $l_{2,3,4} \le 6$ and the
  claim follows from Theorem~\ref{thm:coup}.
\end{proof}

\begin{cor}
  \label{cor:actuated}
  If a joint $h_i$ is connected with multiplicity one to exactly one
  other joint, then the configuration curve can be parametrized by
  $t_i$, or, equivalently, by the rotation angle at $h_i$.
\end{cor}
\begin{proof}
  The assumption is equivalent to $b(i)=1$. If this is true,
  then the coupling map $f_{i-1,i}$ is birational, and its inverse
  is the desired parametrization. 
\end{proof}

\begin{cor}
  If two joints $h_i$, $h_j$ of length $d(i,j) = 1$ are connected with
  multiplicity one to each other and they are not connected to other
  joints, then $t_i = \pm t_j$ holds for all points of the
  configuration curve.
\end{cor}
\begin{proof}
  This follows from Corollary~\ref{cor:actuated} and the fact that the
  configuration curve can be parametrized birationally by $t_i$ or
  by $t_j$. Hence there is a projective equivalence relating $t_i$ and
  $t_j$. This equivalence fixes $\infty$ and takes the zeroes of
  $t_i^2+1$ to the zeroes of $t_j^2+1$. This already implies $t_j=\pm
  t_i$.
\end{proof}

\subsection{More examples}
\label{sec:more-examples}

In this subsection we present three more examples of overconstrained
6R linkages and their bond diagrams. Apparently, the linkages in
Examples~\ref{ex:11} and \ref{ex:12} are new.

\begin{example}
  \label{ex:11}
  We use the method of factorizing motion polynomials~\cite{HSS,HSS2}
  to construct a 6R linkage as follows. First, we choose two arbitrary
  rotation polynomials $h_1$, $h_2$ with non-concurrent axes, say
  $h_1=\qi$ and $h_2=\eps\qi+\qj$. Then we choose a random linear combination
  of $1$ and $h_2$, say $h'_2=1+h_2$, and factor the quadratic motion
  polynomial as $P(t) = (t-h_1)(t-h_2')$. We compute a second factorization
  $P(t) = (t-1-g_1)(t-g_2)$. Then we choose another random linear combination of
  $1$ and $h_2$, say $h''_2=2+h_2$, and factor the motion polynomial
  $Q(t) = (t-g_2)(t-h''_2)$. We get a second factorization
  $Q(t) = (t-2-h_4)(t-h_3)$. Next we choose a random linear combination of $1$ and $h_4$,
  say $h_4'=3+h_4$, and we factor the motion polynomial
  $R(t) = (t-1-g_1)(t-3-h_4)$. We get a second factorization
  $R(t) = (t-3-h_6)(t-1-h_5)$. We obtain a six-bar linkage $L =
  (h_1,h_2,h_3,h_4,h_5,h_6)$ with configuration curve
  \begin{equation*}
    (t_1, t_2, t_3, t_4, t_5, t_6) = \Bigl( t, \frac{t^2-3t+1}{2t-3}, -t, t^2-5t+7, -t+1, -t+3 \Bigr).
  \end{equation*}
  The bonds are
  \begin{equation*}
    \begin{aligned}
      (3\pm\ci,\tfrac{6}{13}\pm\tfrac{9}{13}\ci,-3\mp\ci,\pm\ci,-2\mp\ci,\mp\ci),  \\
      (\pm\ci,-\tfrac{6}{13}\pm\tfrac{9}{13}\ci,\mp\ci,6\mp 5\ci,1\mp\ci,3\mp\ci),  \\
      (2\pm\ci,\pm\ci,-2\mp\ci,\mp\ci,-1\mp\ci,1\mp\ci), \\
      (1\pm\ci,\pm\ci,-1\mp\ci,2\mp 3\ci,\mp\ci,2\mp\ci).
    \end{aligned}
  \end{equation*}
  We have four pairs of conjugate complex bonds. All of them are
  elementary, the bond diagram is given in
  Figure~\ref{fig:more-bond-diagrams}.a).

  The reason why we think that this concrete linkage is new is the
  following. For $i=1,\dots,6$, let $a_i$ be the distance of
  consecutive revolute axes $r_i,r_{i+1}$, let $\alpha_i$ be the angle
  between the axes $r_i,r_{i+1}$, and let $s_i$ be the distance
  between the two points on $r_i$ which are in shortest distance to
  $r_{i-1}$ and $r_{i+1}$. For almost all examples we found in the
  literature, there is at least one equality between the $a_i$, or at
  least one angle $\alpha_i$ which is zero or a right angle, or an
  index $i$ such that $s_i=s_{i+1}=s_{i+2}=0$. There are only two
  exceptions, namely Waldron's double Bennett (see \cite{Dietmaier})
  and the linkage of \cite{HSS,HSS2}. In our example, the numbers
  $a_i$, $i=1,\dots,6$, are pairwise distinct, the angles are neither
  zero nor right angles, and the only vanishing offsets are
  $s_2=s_3=s_5=0$. Hence our example is not a special case of a known
  family, except possibly the double Bennett and the linkages of
  \cite{HSS,HSS2}. But these two families have a different bond
  structure as shown in Figure~\ref{fig:even-more-bond-diagrams}. So,
  our linkage is also not a special case of these linkages.\hfill\eoex
\end{example}

\begin{figure*}
  \centering
  \includegraphics{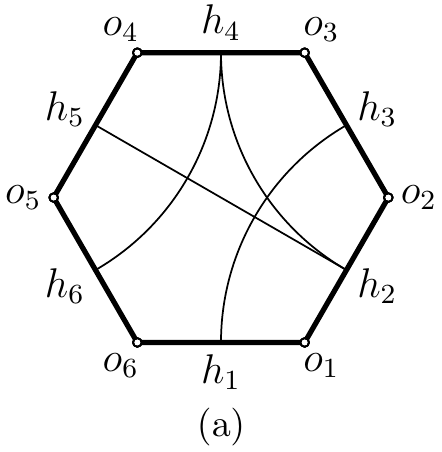}
  \quad
  \includegraphics{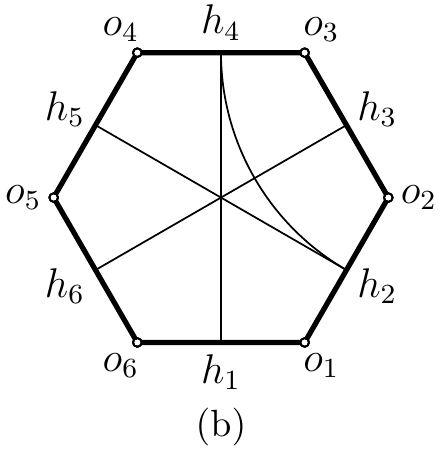}
  \quad
  \includegraphics{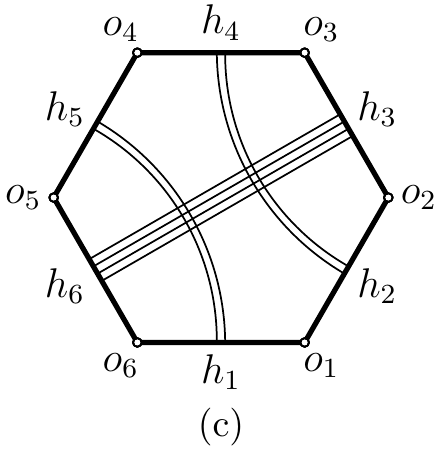}
  \caption{Bond diagrams for the linkages of Examples~\ref{ex:11},
    \ref{ex:12}, and~\ref{ex:13}}
  \label{fig:more-bond-diagrams}
\end{figure*}

\begin{figure*}
  \centering
  \includegraphics{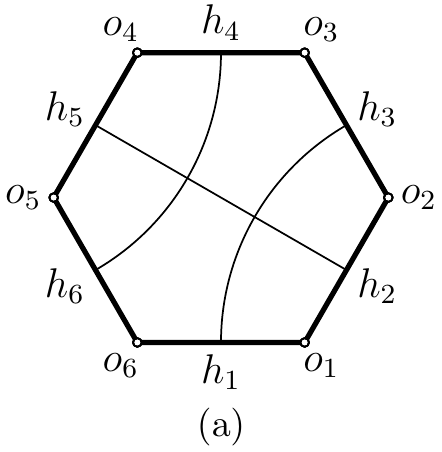}
  \quad
  \includegraphics{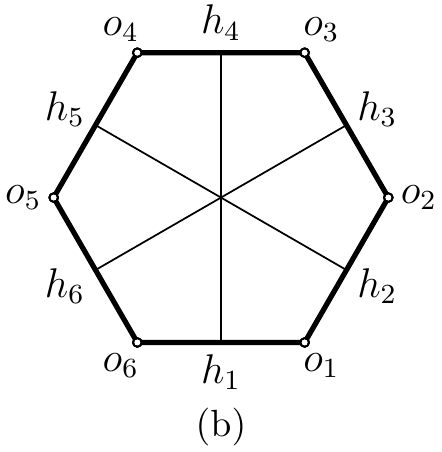}
  \quad
  \caption{Bond diagrams for Waldron's double Bennett linkage and for
    the linkage of \cite{HSS,HSS2}}
  \label{fig:even-more-bond-diagrams}
\end{figure*}

\begin{example}
  \label{ex:12}
  Starting from $h_1,\dots,h_6$ as in the example above, we factor the
  motion polynomial $S(t) = (t+h_1)(t-3-h_6)$ and get a second
  factorization $S(t) = (t-3-h_1')(t-h_6')$. The six-bar linkage $L =
  (h_1',h_2,h_3,h_4,h_5,h_6')$ has configuration curve
  \begin{equation*}
    (t_1, t_2, t_3, t_4, t_5, t_6) =
    (3+t, -\frac{t^2+3t+1}{2t+3}, t, t^2+5t+7, t+1, t).
  \end{equation*}
  The bonds are
  \begin{equation*}
  \begin{aligned}
    (\mp\ci,\tfrac{6}{13}\pm\tfrac{9}{13}\ci,-3\mp\ci,\pm\ci,-2\mp\ci,-3\mp\ci),  \\
    (3\mp\ci,-\tfrac{6}{13}\pm\tfrac{9}{13}\ci,\mp\ci,6\mp 5\ci,1\mp\ci,\mp\ci),  \\
    (1\mp\ci,\pm\ci,-2\mp\ci,\mp\ci,-1\mp\ci,-2\mp\ci), \\
    (2\mp\ci,\pm\ci,-1\mp\ci,2\mp 3\ci,\mp\ci,-1\mp\ci),
  \end{aligned}
  \end{equation*}
  the corresponding bond diagram is shown in
  Figure~\ref{fig:more-bond-diagrams}.b).

  The reason why we think that this example is new is the same as for Example~\ref{ex:12},
  except that now we have only one vanishing offset $s_3=0$ (so we would not need to worry
  about the double Bennett linkage). 
 \hfill\eoex
\end{example}

\begin{example}[Bricard's plane symmetric linkage]
  \label{ex:13}
  We set 
  \begin{equation*}
  \begin{aligned}
  h_1 &= \frac{4}{9}\qi-\frac{17}{27}\eps\qi-\frac{7}{9}\qj-\frac{4}{27}\eps\qj+\frac{4}{9}\qk+\frac{10}{27}\eps\qk,\\
  h_2 &= \frac{3}{7}\qi+\frac{32}{49}\eps\qi+\frac{2}{7}\qj+\frac{12}{49}\eps\qj+\frac{6}{7}\qk-\frac{20}{49}\eps\qk,\\
  h_5 &= \frac{4}{9}\qi+\frac{17}{27}\eps\qi+\frac{7}{9}\qj-\frac{4}{27}\eps\qj-\frac{4}{9}\qk+\frac{10}{27}\eps\qk,\\
  h_4 &= \frac{3}{7}\qi-\frac{32}{49}\eps\qi-\frac{2}{7}\qj+\frac{12}{49}\eps\qj-\frac{6}{7}\qk-\frac{20}{49}\eps\qk,\\
  h_3 &= \qk,\ \ \ h_6=\qj.
  \end{aligned}
  \end{equation*}
  It can be seen that the axes of $h_3$, $h_6$ lie in a plane, and the
  axes of $h_1$, $h_5$ and $h_2$, $h_4$, respectively, are symmetric
  with respect to this plane. Thus, we have an example of Bricard's
  plane symmetric linkage \cite{Baker80},
  \cite[pp.~91--92]{Dietmaier}. The configuration curve has genus one,
  hence it is not parametrizable by polynomials. For the whole
  configuration curve, we have $t_1=-t_5$ and $t_2=-t_4$. One observes
  that the bonds follow the following pattern:
  \begin{gather*}
    (\pm\ci,\ast,\ast,\ast,\mp\ci,\ast), \quad
    (\ast,\pm\ci,\ast,\mp\ci,\ast,\ast), \quad
    (\ast,\ast,\pm\ci,\ast,\ast,\pm\ci), \quad
    (\ast,\ast,\pm\ci,\ast,\ast,\mp\ci)
  \end{gather*}
  (the $\ast$ signs denote complex numbers, all different, with real
  and imaginary part different from zero). The bond diagram is shown
  in Figure~\ref{fig:more-bond-diagrams}.c).
  \hfill\eoex
\end{example}

\section{Classification of closed 5R chains}
\label{sec:classification}

As an application of bond theory, we give a proof of Karger's
classification of overconstrained closed 5R linkages \cite{Karger}.
The main statement is that any non-trivial linkage of this type is a
Goldberg linkage. We also compute the degree of the coupling motions
of Goldberg's linkage.

In the following we suppose that the linkage $L=(h_1,\dots,h_5)$ is a
closed 5R chain with mobility one, which is neither planar nor
spherical, i.e.\ not all five axes are parallel or meet in a point. We
also assume that any two consecutive axes are distinct, and that no
coupling map is constant (for instance $L$ is not a 4R linkage plus
one fixed link). A 5R linkage fulfilling these conditions is called a
non-degenerate 5R linkage.

\begin{lem}
  \label{lem:gt4}
  All coupling dimensions $l_{i,\ldots,j}$ in a non-degenerate 5R
  linkage are greater than four.
\end{lem}

\begin{proof}
  Assume that there exists a coupling dimension which is four, say
  $l_{1,2,3}=4$. Then it follows from Theorem \ref{thm:coup}.c) that
  the axes of $h_1$, $h_2$, $h_3$ intersect in a common point $O$,
  possibly at infinity. The coupling curve $C_{5,4}$ contains only
  rotations around axes through $O$. On the other hand, $L_{5,4}$
  contains only either rotations around $h_4$ and $h_5$ (if these two
  axes are skew) or rotations around axes through the common
  intersection point $O'$ (possibly at infinity), if $h_4$ and $h_5$
  are not skew. Hence $O = O'$ and $L$ is a planar or spherical
  linkage.
\end{proof}

\begin{lem}
  \label{lem:bir}
  For $i \in [n]$ the coupling map $f_{i,i+2}$ is injective.
\end{lem}

\begin{proof}
  Without loss of generality, we assume $i = 5$. The parametrization
  $(\P^1)^2\to X_{1,2}$, $(t_1,t_2)\mapsto (t_1-h_1)(t_2-h_2)$ is
  injective. Similarly, it follows from $l_{5,4,3}>4$ that the
  parametrization $(\P^1)^3\to X_{5,4,3}$, $(t_5,t_4,t_3)\mapsto
  (t_5-h_5)(t_4-h_4)(t_3-h_3)$ is injective. Hence there can be at
  most one configuration $(t_1,t_2,t_3,t_4,t_5)$ that maps into some
  point in the intersection $X_{1,2}\cap X_{5,4,1}$.
\end{proof}

\begin{lem} \label{lem:ne} Let $h_1,\ldots,h_6$ be six half-turns such
  that $L_{1,2,3} = L_{4,5,6} =: L$ and $\dim(L)=6$. Then $h_1=\pm h_4$
  and $h_3=\pm h_6$.
\end{lem}

\begin{proof}
  Let $A\subset\D\H$ be the set of all elements $a$ such that $L$ is
  closed under multiplication with $a$ from the left. Then $A$ is a
  subalgebra, we have $h_1\in A$ and $h_4\in A$, and $A\subset L$
  because $1\in L$. Assume, to that contrary, that $h_1\ne\pm h_4$.
  The only proper subalgebras of $\D\H$ containing two different
  rotations are conjugate to $\SO = \H$ (rotations about one fixed
  point) or to $\SE[2] = \langle 1,\qi,\eps\qj,\eps\qk\rangle$
  (rotations about axes parallel to a fixed direction and translation
  orthogonal to this direction; angled brackets denote linear span).
  The former does not act by left-multiplication on a module of real
  dimension~6. The later acts exactly on one submodule of $\D\H$
  containing~1, namely $\langle 1,\qi,\eps\qj,\eps\qk,\eps,\eps\qi
  \rangle$, which must then be $L$ (up to conjugation). But all
  rotations in this submodule are contained in $A$, hence
  $h_1,h_2,h_3\in A$ and $A=L$, which is a contradiction.
\end{proof}

\begin{lem} \label{lem:main}
  If $l_{1,2,3}=l_{3,4,5}=6$, then $b(1)=b(2)=b(4)=b(5)=1$ and $b(3)=2$.
\end{lem}

\begin{proof}
  Let $L:=L_{1,2,3}\cap L_{5,4,3}$. The dimension of $L$ is even,
  because it is an $L_3$-right vectorspace. By Lemma~\ref{lem:ne} the
  spaces $L_{1,2,3}$ and $L_{5,4,3}$ are different. Hence $\dim(L)\le
  4$. On the other hand, $\dim(L)\ge l_{1,2,3}+l_{6,5,4}-8$. Hence, we
  have $\dim(L)=4$.

  First we prove that $d(3,5)=2$. By Theorem~\ref{thm:bond}.d we have
  $d(3,5) = b(4) + b(5) \ge 1 + 1 = 2$ ($b(i) > 0$ because all joints
  move). Assume, to the contrary, that $d(3,5) \ge 3$. Then $C_{3,5}$
  is a curve of degree at least three, because $\deg f_{3,5} = 1$ by
  Lemma~\ref{lem:bir}. On the other hand, the ideal of $C_{3,5}$ is
  generated by linear and quadratic equations, because
  $C_{3,5}=X_{1,2,3}\cap X_{5,4}$ and the ideals of $X_{1,2,3}$ and of
  $X_{5,4}$ are generated by linear and quadratic equations. Hence,
  $C_{3,5}$ is not a plane curve, because otherwise the degree of
  $C_{3,5}$ would be at most two. Because of $C_{3,5} \subset L' :=
  L_{5,4} \cap L_{1,2,3}$, this implies $\dim L' = 4$ and, thus, $L =
  L'$. But then $L_{5,4} \subseteq L_{1,2,3}$. If we multiply both
  sides with $h_3$ from the right, we get $L_{5,4,3} \subseteq
  L_{1,2,3}$. Both spaces have the same dimension, hence $L_{1,2,3} =
  L_{5,4,3}$ -- a contradiction. This proves $d(3,5)=2$ and also
  $b(4)=b(5)=1$. Applying the same argument for the linkage
  $(h_5,h_4,h_3,h_2,h_1)$, we get $b(1)=b(2)=1$.

  It remains to be shown that $b(3)=2$. By the triangle inequality,
  $b(3)=d(2,3)\le d(2,5)+d(3,5)=b(1)+b(2)+b(4)+b(5)=4$. By
  Theorem~\ref{thm:bond}.c, $d(2,5)+d(2,3)+d(3,5)$ is even, hence the
  bond length $b(3)$ is even. Clearly, $0 < b(3) \le 4$. If $b(3)=4$,
  then $d(1,3)=b(2)+b(3)=5$, contradicting the triangle inequality
  $d(1,3)\leq d(1,5)+d(3,5) = b(1) + b(4) + b(5) = 3$. Thus, $b(3)=2$
  and the proof is finished.
\end{proof}

\begin{lem} \label{lem:5r}
  Let $L$ be a non-degenerate 5R linkage. Then exactly one of its
  joint lengths is equal to two, and all others are equal to one.
\end{lem}

\begin{proof}
  By Lemma~\ref{lem:gt4}, the numbers $l_{i,i+1,i+2}$ can only be $6$
  or $8$ for $i=1,\dots,5$ (the indices are labeled modulo 5). Because
  $5$ is an odd number, there exists an index $i$ such that
  $l_{i-2,i-1,i}=l_{i,i+1,i+2}$. Without loss of generality, we may
  assume $i=3$. We distinguish two cases.

  Case~1: $l_{1,2,3}=l_{3,4,5}=8$. By
  Theorem~\ref{thm:bond}.d and Corollary~\ref{cor:impl}, $l_{1,2,3}=8$
  implies $d(5,3)=d(5,1)+d(1,2)+d(2,3)=b(1)+b(2)+b(3)=b(4)+b(5)$.
  Similarly, $l_{3,4,5}=8$ implies $b(3)+b(4)+b(5)=b(1)+b(2)$. Hence
  $b(3)=0$, a contradiction.

  Case~2: $l_{1,2,3}=l_{3,4,5}=6$. Then Lemma~\ref{lem:main} applies.
\end{proof}

We are already in a position to state a new result on overconstrained
5R linkages:
\begin{thm}
  The coupling motions of a non-degenerate 5R linkage can be
  parametrized by four of the five joint angles. Its coupling curves
  (that is, the relative motions as curves on the Study quadric) are
  plane conics and twisted cubics.
\end{thm}

\begin{proof}
  The coupling curves can be parametrized by the angles at all four
  joints of length one. The coupling curves $C_{i-1,i+2}$ have degree
  $d(i,i+2)=b(i+1)+b(i+2)$ for $i=1,\dots,5$ (with cyclic numbering of
  indices), and this is two or three. Since the ideals of the coupling
  curves are generated by linear and quadratic forms, they can only be
  plane conics or twisted cubics.
\end{proof}

The main result of this section is
\begin{thm}
  \label{thm:goldberg}
  Every non-degenerate 5R linkage is a Goldberg linkage.
\end{thm}

\begin{proof}
  Denote the linkage by $L=(h_1,h_2,h_3,h_4,h_5)$. By
  Lemma~\ref{lem:5r}, there is one joint, say $h_3$, of length two,
  that is $b(3) = 2$. The coupling curve $C_{1,4}$ is a twisted cubic,
  in particular it is a rational curve of degree three. We fix a cubic
  parametrization $\phi:t\mapsto P(t)$ of degree three and apply the
  synthesis method of \cite{HSS,HSS2} for synthesizing open 3R chains
  that are parametrized linearly and that produce the motion $\phi$.

  By general results of \cite{HSS,HSS2}, the relative motion $C_{1,4}$
  admits parametrizations
  \begin{gather*}
    (t-h''_3)(t-h_3')(t-h_2'),\quad
    (t-h'_4)(t-h_5')(t-h_1'),\quad
    (t-h''_3)(t-h'_6)(t-h_1')
  \end{gather*}
  with $h'_i \in L_i$ for $i=1,\ldots,6$ and $h''_3 \in L_3$ such that
  $(h'_1,h'_2,h'_3,h'_6)$ and $(h''_3,h'_4,h'_5,h'_6)$ is a Bennett
  quadruple. The original 5R linkage can be constructed by composition
  of these two Bennett linkages, with the common axes $h_3$, $h_6$,
  and subsequent removal of the joint at $h_6$. This is exactly
  Goldberg's construction \cite{Goldberg,Wohlhart} which we have shown
  to be necessary for a non-degenerate 5R linkage. The existence of Goldberg's linkage is
  well-known and easy to see.
\end{proof}

\section{Conclusion}
\label{sec:conclusion}

This article featured a rigorous introduction of bonds, connection
numbers and bond multiplicities. Bond theory is a new tool for
analyzing overconstrained linkages with one degree of freedom and one
can hope to gain new insight into their behavior. In order to
demonstrate the usefulness of bond theory, we gave a new proof of
Karger's classification theorem for overconstrained 5R chains. Note
that bond theory can provide necessary conditions for overconstrained
linkages but neither their sufficiency nor existence of linkages with
a particular bond structure is automatically implied (compare also our
proof of Theorem~\ref{thm:goldberg}). In a next step, we plan to work
out the bond structure for overconstrained 6R chains, both known and
new. Some examples have already been given in this paper.

\section*{Acknowledgements}

We would like to thank Zijia Li (RICAM Linz) for helping us with
computations of examples.
This research was supported by the Austrian Science Fund (FWF) under
grants I~408-N13 and DK~W~1214-N15 and the National Science Foundation
Research Grant No. 77476.


\end{document}